\documentclass[10pt, reqno]{amsart}

\usepackage[
paper=a4paper,
text={147mm,230mm},centering
]{geometry}
\usepackage{amssymb,amsmath}
\usepackage{url}
\iftrue
\makeatletter
\def\@settitle{%
  \vspace*{-20pt}
  \begin{flushleft}%
    \baselineskip14\p@\relax
    \normalfont\bfseries\LARGE
    \@title
  \end{flushleft}%
}

\def\@setaddresses{\par
  \nobreak \begingroup\raggedright
  \small
  \def\author##1{\nobreak\addvspace\smallskipamount}%
  \def\\{\unskip, \ignorespaces}%
  \interlinepenalty\@M
  \def\address##1##2{\begingroup
    \par\addvspace\bigskipamount\noindent
    \@ifnotempty{##1}{(\ignorespaces##1\unskip) }%
    {\ignorespaces##2}\par\endgroup}%
  \def\curraddr##1##2{\begingroup
    \@ifnotempty{##2}{\nobreak\noindent\curraddrname
      \@ifnotempty{##1}{, \ignorespaces##1\unskip}\/:\space
      ##2\par}\endgroup}%
  \def\email##1##2{\begingroup
    \@ifnotempty{##2}{\smallskip\nobreak\noindent E-mail address%
      \@ifnotempty{##1}{, \ignorespaces##1\unskip}\/:\space
      \ttfamily##2\par}\endgroup}%
  \def\urladdr##1##2{\begingroup
    \def~{\char`\~}%
    \@ifnotempty{##2}{\nobreak\noindent\urladdrname
      \@ifnotempty{##1}{, \ignorespaces##1\unskip}\/:\space
      \ttfamily##2\par}\endgroup}%
  \addresses
  \endgroup
  \global\let\addresses=\@empty
}
\def\@setabstracta{%
    \ifvoid\abstractbox
  \else
    \skip@25\p@ \advance\skip@-\lastskip
    \advance\skip@-\baselineskip \vskip\skip@
    \box\abstractbox
    \prevdepth\z@ % because \abstractbox is a vtop
    \vskip-10pt
  \fi
}
\renewenvironment{abstract}{%
  \ifx\maketitle\relax
    \ClassWarning{\@classname}{Abstract should precede
      \protect\maketitle\space in AMS document classes; reported}%
  \fi
  \global\setbox\abstractbox=\vtop \bgroup
    \normalfont\small
    \list{}{\labelwidth\z@
      \leftmargin0pc \rightmargin\leftmargin
      \listparindent\normalparindent \itemindent\z@
      \parsep\z@ \@plus\p@
      
    }%
    \item[\hskip\labelsep\bfseries\abstractname.]%
}{%
  \endlist\egroup
  \ifx\@setabstract\relax \@setabstracta \fi
}
\def\section{\@startsection{section}{1}%
  \z@{-1.2\linespacing\@plus-.5\linespacing}{.8\linespacing}%
  {\normalfont\bfseries\large}}
\def\subsection{\@startsection{subsection}{2}%
  \z@{-.8\linespacing\@plus-.3\linespacing}{.3\linespacing\@plus.2\linespacing}%
  {\normalfont\bfseries}}
\def\subsubsection{\@startsection{subsubsection}{3}%
  \z@{.7\linespacing\@plus.1\linespacing}{-1.5ex}%
  {\normalfont\itshape}}
\def\@secnumfont{\bfseries}
\makeatother
\fi %\iftrue/false for amsart.cls hack
\usepackage{fancyhdr}
\usepackage{amssymb}
\usepackage[all]{xy}
\usepackage{graphicx,color,float}
\usepackage[bookmarks]{hyperref}
\usepackage{pinlabel}
\usepackage[textwidth=1.3in,color=yellow]{todonotes}
\usepackage{amsmath}
\usepackage{pb-diagram,pb-xy}
\textwidth=\paperwidth \advance\textwidth by-2\oddsidemargin
\advance\oddsidemargin by-1in
\evensidemargin=\oddsidemargin
\calclayout

\def\Z{\mathbb{Z}}

\def\+{\oplus}

\newcommand{\genlegendre}[4]{%
  \genfrac{(}{)}{}{#1}{#3}{#4}%
  \if\relax\detokenize{#2}\relax\else_{\!#2}\fi
}

\theoremstyle{plain}
\newtheorem{theorem}{Theorem}[section]
\newtheorem{proposition}[theorem]{Proposition}
\newtheorem{lemma}[theorem]{Lemma}
\newtheorem{corollary}[theorem]{Corollary}

\theoremstyle{definition}

\newtheorem{remark}[theorem]{Remark}

\newtheorem*{sliceribbon}{The slice-ribbon conjecture}
\newtheorem{conjecture}{Conjecture}

\makeatletter
\newcommand{\shortxra}[2][]{\ext@arrow 0359\rightarrowfill@{#1}{#2}}
\def\longrightarrowfill@{\arrowfill@\relbar\relbar\longrightarrow}
\newcommand{\longxra}[2][]{\ext@arrow 0359\longrightarrowfill@{#1}{#2}}

\makeatother
\begin{document}

\title[]{Non-slice 3-stranded pretzel knots}

\author{Min Hoon Kim}
\address{Chonnam National University, 77, Yongbong-ro, Buk-gu, Gwangju, Republic of Korea}
\email{minhoonkim@jnu.ac.kr}

\author{Changhee Lee}
\address{Department of Mathematics, POSTECH, Pohang 37673, Republic of Korea}
\email{leechanghee@postech.ac.kr}

\author{Minkyoung Song}
\address{Center for Geometry and Physics, Institute for Basic Science (IBS), Pohang
37673, Republic of Korea}
\email{mksong@ibs.re.kr}

\thanks{The
first named author was partly supported by NRF grant 2019R1A3B2067839. The third named author was supported by the Institute for Basic Science (IBS-R003-D1).}

\def\subjclassname{\textup{2020} Mathematics Subject Classification}
\expandafter\let\csname subjclassname@1991\endcsname=\subjclassname
\expandafter\let\csname subjclassname@2000\endcsname=\subjclassname
%\keywords{}

\begin{abstract}
Greene-Jabuka and Lecuona confirmed the slice-ribbon conjecture for 3-stranded pretzel knots except for an infinite family $P(a,-a-2,-\frac{(a+1)^2}{2})$ where $a$ is an odd integer greater than $1$.  Lecuona and Miller showed that $P(a,-a-2,-\frac{(a+1)^2}{2})$ are not slice unless $a\equiv 1, 11, 37, 47, 59 \pmod{60}$. In this note, we show that four-fifths of the remaining knots in the family are not slice. 
\end{abstract}

\maketitle

\section{Introduction}
A knot $K$ in $S^3$ is \emph{slice} if $K$ bounds a smoothly embedded disk in the 4-ball $D^4$. A knot $K$ in $S^3$ is \emph{ribbon} if $K$ bounds an immersed disk in $S^3$ having only ribbon singularities. It is easy to see that ribbon knots are slice since we can push such an immersed disk in $S^3$ around the ribbon singularities  to get an embedded disk in $D^4$. A famous open question posed by Fox \cite{Fox:1961-1} asks whether all slice knots are ribbon or not. This question is usually called the \emph{slice-ribbon conjecture}.

\begin{sliceribbon}[{\cite{Fox:1961-1}}] Every slice knot is ribbon.
\end{sliceribbon}

In \cite{Lisca:2007-1}, Lisca confirmed the slice-ribbon conjecture for 2-bridge knots. The slice-ribbon conjecture has been extensively studied for pretzel knots.  In \cite{Greene-Jabuka:2011-1}, Greene and Jabuka confirmed the slice-ribbon conjecture for 3-stranded pretzel knots $P(p,q,r)$ with odd $p, q, r$. (Figure~\ref{figure:Pretzel} depicts a diagram of a pretzel knot $P(p,q,r)$ where $p,q,r$ in the boxes represent the numbers of half twists.) 

\begin{figure}[htb!]
\includegraphics[scale =0.9] {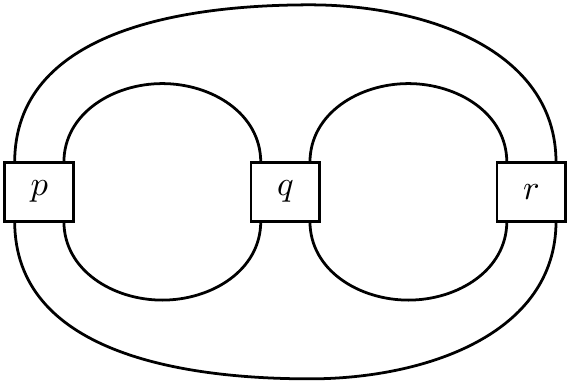}
\caption{3-stranded pretzel knot $P(p,q,r)$.}
\label{figure:Pretzel}
\end{figure}

Lecuona \cite{Lecuona:2015-1} confirmed the slice-ribbon  conjecture for all 3-stranded pretzel knots $P(p,q,r)$ except for infinitely many 3-stranded pretzel knots of the form $P(a,-a-2,-\frac{(a+1)^2}{2})$ for $a\geq 3$ odd. There are further interesting results on the slice-ribbon conjecture for general pretzel knots and links (for example, see \cite{Long:2014-1,Bryant:2017-1,Aceto-Kim-Park-Ray:2018-1,Kishimoto-Shibuya-Tsukamoto:2020-1}). 

Regarding the remaining 3-stranded pretzel knots $P(a,-a-2,-\frac{(a+1)^2}{2})$ with odd $a\geq 3$, Lecuona conjectured that none of them are slice even topologically. (Recall that a knot $K$ in $S^3$ is \emph{topologically slice} if $K$ bounds a locally flat, embedded disk in $D^4$.) 
 
 \begin{conjecture}[{\cite[Remark 2.3]{Lecuona:2015-1}}]\label{conjecture:Lecuona}
 For every odd integer $a\geq 3$, the $3$-stranded pretzel knot $P(a,-a-2,-\frac{(a+1)^2}{2})$ is not topologically slice.
 \end{conjecture} 
 
 Conjecture~\ref{conjecture:Lecuona} is of interest since it implies that the slice-ribbon conjecture is true for all 3-stranded pretzel knots.  Lecuona showed that Conjecture~\ref{conjecture:Lecuona} is true unless $a\equiv 1,11, 37, 47, 49, 59 \pmod{60}$ and $a\geq 1599$. Later Miller showed in  \cite[Theorem 5.1]{Miller:2017-1} that  Conjecture~\ref{conjecture:Lecuona} is true when $a\equiv 49\pmod{60}$. 
 
 More interestingly,  confirming Conjecture~\ref{conjecture:Lecuona} will lead us to classify all 3-stranded pretzel knots which are topologically slice. In fact, for a 3-stranded pretzel knot $K$ which is not of the form $P(a,-a-2,-\frac{(a+1)^2}{2})$ with $a\geq 3$ odd, Miller showed in \cite[Theorems 1.2 and 1.6]{Miller:2017-1} that  $K$ is topologically slice if and only if either $\Delta_K(t)=1$ or $K$ is ribbon. 

 Conjecture~\ref{conjecture:Lecuona} has remained open when $a\equiv 1,11,37,47, 59\pmod{60}$. Our main result is to show that four-fifths of the remaining knots of Conjecture~\ref{conjecture:Lecuona} are not slice. 
\begin{theorem}\label{thm:main-introduction} Let $a\geq 3$ be an odd integer. 
If $a\equiv 11, 47, 59\pmod{60}$ or $a\equiv 37, 61 \pmod{120}$, then the $3$-stranded pretzel knot $P(a,-a-2,-\frac{(a+1)^2}{2})$ is not topologically slice. 
\end{theorem}

\begin{proof}[Proof of Theorem~\ref{thm:main-introduction}]
In Theorems~\ref{thm:main} and ~\ref{thm:main2} given below, we show that $P(a,-a-2,-\frac{(a+1)^2}{2})$ is not topologically slice if $a\equiv 3\pmod{4}$ or $a\equiv 5\pmod{8}$.  Theorem~\ref{thm:main-introduction} immediately follows.
\end{proof}

Hence Conjecture~\ref{conjecture:Lecuona} remains open for $a\equiv 1, 97\pmod{120}$. By the aforementioned results of Lecuona \cite{Lecuona:2015-1} and Miller \cite{Miller:2017-1}, we have the following corollary.

\begin{corollary}\label{cor:1}
Let $K$ be a $3$-stranded pretzel knot which is not of the form $P(a,-a-2,-\frac{(a+1)^2}{2} )$ for $a\equiv 1, 97\pmod{120}$. Then $K$ is smoothly slice if and only if $K$ is ribbon. 
\end{corollary}

By \cite[Theorems 1.2 and 1.6]{Miller:2017-1}, Theorem~\ref{thm:main-introduction} immediately gives  the following corollary. 
\begin{corollary}Let $K$ be a $3$-stranded pretzel knot which is not of the form $P(a,-a-2,-\frac{(a+1)^2}{2} )$ for $a\equiv 1, 97\pmod{120}$. Then $K$ is topologically slice if and only if  either $K$ is ribbon or $\Delta_K(t)=1$. 
\end{corollary}

%\begin{corollary}\label{cor:general}
%Let $\mathcal{S}:=\{\{a,-a-2,-\frac{(a+1)^2}{2}\}\mid a\geq 3, a \equiv 1,97 \pmod{120}\}$. 
%The pretzel knots of the form $P(p_1,\ldots,p_n)$ with odd $n$ and one even $p_i$ are not topologically slice if the set of parameters $\{p_1,\ldots,p_n\}$ is not $\{q_0, q_1, -q_1, \ldots, q_m, -q_m\}$, and any subset is not in $\mathcal{S}$.
%\end{corollary}

%\begin{proof}[Proof of Corollary~\ref{cor:general}]
%By \cite[Theorem~1.1]{Lecuona:2015-1}, it is enough to check the non-sliceness of the pretzel knots $P(p_1,\ldots,p_n)$ with $\{p_1,\ldots,p_n\}=\{a,-a-2,-\frac{(a+1)^2}{2},q_1, -q_1, \ldots, q_m, -q_m\}$ for $a \geq 3$ and $a\not\equiv 1,97 \pmod{120}$. Note that all the pretzel knots with the same set of parameters have the same Alexander polynomial since the Alexander polynomial is invariant under mutation.
%From \cite[Theorem~1.1]{Shibuya-Tsukamoto-Uchida:2020}, we have
%\[\Delta_{P(a,-a-2,-\frac{(a+1)^2}{2},q_1, -q_1, \ldots, q_m, -q_m)}(t) \doteq \Delta_a(t) \prod_{q_i} (1-t+\cdots+ t^{q_i-1})^2.
%\]
%Hence, $\Delta_{P(a,-a-2,-\frac{(a+1)^2}{2},q_1, -q_1, \ldots, q_m, -q_m)}(t)$ admits no Fox-Milnor factorization for $a\geq 3$ and $a\not\equiv 1,97 \pmod{120}$, which are the cases that $\Delta_a(t)$ admits no Fox-Milnor factorization.
%\end{proof}

\section{Alexander polynomial of $P(a,-a-2,-\frac{(a+1)^2}{2})$}
In this section, we  recall elementary facts on the Alexander polynomial of a slice knot and collect observations of Lecuona on  the Alexander polynomials of 3-stranded pretzel knots $P(a,-a-2,-\frac{(a+1)^2}{2})$ following \cite[pp.\ 2151--2152]{Lecuona:2015-1}.

Fox and Milnor observed in \cite{Fox-Milnor:1966-1} that the Alexander polynomial of  a slice knot can be factored in a special form.

\begin{theorem}[{\cite{Fox-Milnor:1966-1}}]\label{thm:Fox-Milnor}
The Alexander polynomial $\Delta_K(t)$ of a topologically slice knot $K$ satisfies 
\[\Delta_K(t) \doteq f(t) f(t^{-1}) \]
for some $f(t)\in \Z[t^{\pm 1}]$ and $\doteq$ means equality up to multiplication by a unit of $\Z[t^{\pm 1}]$. 
\end{theorem}

Let $R$ be either $\Z$ or $\Z_p$ for some prime $p$. We say that a Laurent polynomial $\Delta(t)\in R[t^{\pm 1}]$ admits a \emph{Fox-Milnor factorization} if \[\Delta(t) \doteq f(t) f(t^{-1}) \]
for some $f(t)\in R[t^{\pm 1}]$ and $\doteq$ means equality up to multiplication by a unit of $R[t^{\pm 1}]$. By Theorem~\ref{thm:Fox-Milnor}, for a topologically slice knot $K$,  the Alexander polynomial $\Delta_K(t)$ and its mod $p$ reduction $\overline{\Delta}_K^p(t)$ admit Fox-Milnor factorizations.

\begin{remark}\label{rmk:SR}A polynomial $f(t)\in R[t^{\pm 1}]$ is \emph{self-reciprocal} if $f(t)=t^{\deg f}f(t^{-1})$. The Alexander polynomials of knots and the cyclotomic polynomials are self-reciprocal. Irreducible factors of a self-reciprocal polynomial are all self-reciprocal or come in \emph{reciprocal pairs}, that is, if $g(t)\mid f(t)$ and $g(t)$ is not self-reciprocal then $t^{\deg g}  g(t^{-1})$ is also a factor of $f(t)$. It follows that a self-reciprocal polynomial admits a Fox-Milnor factorization if and only if the multiplicities of its irreducible self-reciprocal factors are all even.
\end{remark}

For an odd integer $a\geq 3$, let $P_a$ be the 3-stranded pretzel knot $P(a,-a-2,-\frac{(a+1)^2}{2})$. In \cite[Appendix]{Lecuona:2015-1}, Lecuona computed the Alexander polynomial $\Delta_{P_a}(t)$: 
\begin{align*}
\textstyle  \Delta_{P_a}(t) 
   &\textstyle \doteq \frac{(t^{a+2}+1)(t^a+1)}{(t+1)^2}-\frac{(a+1)^2}{4}t^{a-1}(t-1)^2\\
 &=\textstyle  \prod\limits_{\substack{d \mid a \\ d \neq 1}} \Phi_{d}(-t) \prod\limits_{\substack{d \mid a+2 \\ d \neq 1}} \Phi_{d}(-t)-\frac{(a+1)^2}{4}t^{a-1}(t-1)^2
\end{align*}
where $\Phi_{d}(t)$ is the $d$th cyclotomic polynomial.
For a prime $p$ dividing $\frac{a+1}{2}$, the mod $p$ reduction $\overline{\Delta}_{P_a}^p(t)$ of $\Delta_{P_a}(t)$ is a product of the mod $p$ reductions of some cyclotomic polynomials:

\begin{equation}\label{eqn:Alexander}
  \textstyle\overline{\Delta}_{P_a}^{p}(t) \doteq \prod\limits_{\substack{d \mid a \\ d \neq 1}} \overline{\Phi}_{d}^p(-t) \prod\limits_{\substack{d \mid a+2 \\ d \neq 1}} \overline{\Phi}_{d}^p(-t) \in \Z_p[t^{\pm 1}].
\end{equation}

Suppose that $P_a$ is topologically slice. Then both $\Delta_{P_a}(t)$ and $\overline{\Delta}_{P_a}^p(t)$ admit Fox-Milnor factorizations by Theorem~\ref{thm:Fox-Milnor}. Since $\gcd(p,a)=\gcd(p,a+2)=\gcd(a, a+2)=1$, all the irreducible factors of $\overline{ \Phi}_d^p(-t)$ appearing in \eqref{eqn:Alexander} are distinct and have multiplicity 1. Let $N_d^p$ be the number  of  irreducible factors of $\overline{\Phi}_d^p$. By Remark~\ref{rmk:SR}, $N_d^p$ is even for any integer $d> 1$ dividing either $a$ or $a+2$.  Hence we have the following proposition which is due to Lecuona (see \cite[pp.\ 2151--2152]{Lecuona:2015-1}).

\begin{proposition}[{\cite[pp.\ 2151--2152]{Lecuona:2015-1}}]\label{prpn:Lecuona-observation}
For an odd integer $a\geq 3$, let $P_a$ be the $3$-stranded pretzel knot $P(a,-a-2,-\frac{(a+1)^2}{2})$. 
Let $p$ be a prime dividing $\frac{a+1}{2}$ and $d>1$ be an integer dividing either $a$ or $a+2$. If $P_a$ is topologically slice, then the following holds.
\begin{enumerate}
\item\label{item:Lecuona-observation-1} $\Delta_{P_a}(t)$ admits a Fox-Milnor factorization.
\item\label{item:Lecuona-observation-2} $\overline{\Delta}_{P_a}^p(t)$ admits a Fox-Milnor factorization.
\item\label{item:Lecuona-observation-3} $\overline{\Phi}_d^p(-t)$ has no irreducible self-reciprocal factor. 
\item\label{item:Lecuona-observation-4} $N_d^p$ is even.
\end{enumerate}
\end{proposition}

Lecuona showed that $N_d^p$ is odd for some such $p$ and $d$ when $a\equiv 3,7,9 \pmod{12}$ and $a\equiv 13,23,25,35\pmod{60}$ to obtain the non-sliceness of $P_a$. So did Miller \cite{Miller:2017-1} when $a\equiv 49\pmod{60}$. For the case $a\equiv 5 \pmod{12}$, Lecuona found an irreducible self-reciprocal factor $t^2+t+1$ of $\bar\Delta_{P_a}^2(t)$, which has multiplicity 1.

The remainder of this paper is organized as follows. In Section~\ref{section:method1}, we interpret the condition \eqref{item:Lecuona-observation-4} in terms of a Legendre symbol and show that \eqref{item:Lecuona-observation-4} is not satisfied if either $a\equiv 3,5\pmod{8}$ or $a\equiv 5\pmod{12}$. In Section~\ref{sec:SR}, we discuss two  conditions which are equivalent to \eqref{item:Lecuona-observation-3} and show that \eqref{item:Lecuona-observation-3} is not satisfied if $a\equiv 3\pmod{4}$. In Section~\ref{sec:examples}, we mention some knots $P_a$ satisfying \eqref{item:Lecuona-observation-2}, \eqref{item:Lecuona-observation-3} and \eqref{item:Lecuona-observation-4}.

\section{Method 1: Using the Legendre symbol $\bigl(\frac{p}{d}\bigr)$}\label{section:method1}
The goal of this section is to prove Theorem~\ref{thm:main} which states that $P_a$ is not topologically slice if $a\equiv 3,5 \pmod{8}$ or $a\equiv 5 \pmod{12}$.  For this purpose, we will observe in Proposition~\ref{prpn:Legendre} that determining the parity of $N_d^p$ is equivalent to computing the Legendre symbol $\bigl(\frac{p}{d}\bigr)$.

For the reader's convenience, we recall the definition and properties of Legendre symbols. For an integer $n$ and an odd prime $d$, the \emph{Legendre symbol} $\bigl(\frac{n}{d}\bigr)$ is defined by 
\[\bigl(\tfrac{n}{d}\bigr) = \left\{ \begin{array}{ll} 
-1 & \textrm{if $n$ is not a square modulo $d$}, \\
	0 & \textrm{if $n$ is divisible by $d$}, \\
	
	1  & \textrm{if $n$ is a nonzero square modulo $d$}.
	 \end{array}\right. \]
We need the following standard properties on Legendre symbols (for example, see \cite[pp.\ 6--7]{Serre}).
\begin{lemma}\label{lemma:QR}
Let $d,p$ be odd primes and $a,b$ be integers.
 \begin{enumerate}
   \item $\bigl(\frac{a}{d}\bigr) = \bigl(\frac{b}{d}\bigr)$ if $a\equiv b \pmod{d}$.
   \item $\bigl(\frac{ab}{d}\bigr) = 
   \bigl(\frac{a}{d}\bigr)\bigl(\frac{b}{d}\bigr)$. 
   In particular, $\bigl(\frac{a^{2}}{d}\bigr) = 1$.
   \item $ \bigl(\frac{p}{d}\bigr)=(-1)^{\frac{(p-1)(d-1)}{4}} \bigl(\frac{d}{p}\bigr)$.
   \item $\bigl(\frac{-1}{d}\bigr) = (-1)^{\frac{d-1}{2}}$.
   \item $\bigl(\frac{2}{d}\bigr) = (-1)^{\frac{d^{2}-1}{8}}$.
 \end{enumerate}
\end{lemma}

\begin{proposition}\label{prpn:Legendre} Let $p$ and $d$ be distinct primes and $d$ is odd. Then the number $N_d^p$ of irreducible factors of $\overline{\Phi}_d^p(t)$ is even if and only if the Legendre symbol $\bigl(\frac{p}{d}\bigr)=1$.
\end{proposition}

\begin{proof}Note that $N_d^p$ is equal to the quotient of Euler's totient function $\varphi(d)$ divided by the multiplicative order of $p$ modulo $d$  and hence $N_d^p=|\Z_d^\times / p \Z_d^\times|$.

Suppose that $N_d^p=|\Z_d^\times / p\Z_d^\times |$ is even.  Since $d$ is prime, $\Z_d^\times$ is cyclic. Let $y$ be a generator of $\Z_d^\times$ and $m$ be an integer such that $y^m\equiv p \pmod{d}$. Since $y^m$ is trivial in $\Z_d^\times/p\Z_d^\times$ and $y$ also generates $\Z_d^\times/p\Z_d^\times$,  $m$ is a multiple of $N_d^p$. It follows that $m$ is even and  $(y^{m/2})^2\equiv p \pmod{d}$ or $\bigl(\frac{p}{d}\bigr)=1$.

Suppose that $\bigl(\frac{p}{d}\bigr)=1$. Choose an integer $x\in \Z_d^\times$  such that $x^2\equiv p \pmod{d}$. Since $x^2$ is trivial in $\Z_d^\times/p\Z_d^\times$, the order of $x$ in $\Z_d^\times/p\Z_d^\times$ is either 1 or 2. In the latter case,  $N_d^p$ is automatically even since $N_d^p$ is equal to the order of $\Z_d^\times/p\Z_d^\times$. Suppose that $x$ is trivial in $\Z_d^\times/p\Z_d^\times$. Then $x\equiv p^m \pmod{d}$ for some $m\in\Z$. Since $x^2\equiv p \pmod{d}$ and $x\equiv p^m \pmod{d}$, 
$p^{2m-1}$ is trivial in $\Z_d^\times$.  So $|p\Z_d^\times|$ is a divisor of $2m-1$ so $|p\Z_d^\times|$ is odd. Since $d$ is an odd prime, $|\Z_d^\times|$ is even and hence
\[N_d^p=|\Z_d^\times/p\Z_d^\times |=\frac{|\Z_d^\times|}{|p\Z_d^\times|}\]
 is even.
\end{proof}

We are ready to prove the main result of this section.

\begin{theorem}\label{thm:main}
For an odd integer $a\geq 3$, let $P_a$ be the $3$-stranded pretzel knot $P(a,-a-2,-\frac{(a+1)^2}{2})$. If $a\equiv 3,5 \pmod{8}$ or $a\equiv 5 \pmod{12}$, then the Alexander polynomial $\Delta_{P_a}(t)$ does not have a Fox-Milnor factorization and hence $P_a$ is not topologically slice.
\end{theorem}

\begin{proof}
Assume that $a$ is an integer with $a\equiv 3 \pmod{8}$. Then $2$ is a prime divisor of~$\frac{a+1}{2}$ and $a$ has a prime divisor $d$ such that $d\equiv 3\pmod{8}$ or $d\equiv 5\pmod{8}$. By Lemma~\ref{lemma:QR}(5), we can easily compute the Legendre symbol $\bigl(\frac{2}{d}\bigr)$ as follows: 
\[\bigl(\tfrac{2}{d}\bigr)=(-1)^{(d^2-1)/8}=-1.\]
By Proposition~\ref{prpn:Legendre}, $N_d^2$ is odd and $\Delta_{P_a}(t)$ does not have a Fox-Milnor factorization.

Assume that $a$ is an integer such that $a\equiv 5 \pmod{12}$. Then $3$ is a prime divisor of $\frac{a+1}{2}$ and $a$ has a prime divisor $d$ such that $d\equiv 5\pmod{12}$ or $d\equiv 7\pmod{12}$. By the quadratic reciprocity law given in Lemma~\ref{lemma:QR}(3), 
\[\bigl(\tfrac{3}{d}\bigr)=\bigl(\tfrac{d}{3}\bigr)(-1)^{(d-1)/2}.\]
When $d\equiv 5 \pmod{12}$, $\bigl(\frac{d}{3}\bigr)=\bigl(\frac{2}{3}\bigr)=-1$ by Lemma~\ref{lemma:QR}(1), and $(-1)^{(d-1)/2}=1$. When $d\equiv 7\pmod{12}$, $\bigl(\frac{d}{3}\bigr)=\bigl(\frac{1}{3}\bigr)=1$ and $(-1)^{(d-1)/2}=-1$. Hence 
\[\bigl(\tfrac{3}{d}\bigr)=\bigl(\tfrac{d}{3}\bigr)(-1)^{(d-1)/2}=-1.\]
By Proposition~\ref{prpn:Legendre}, $N_d^3$ is odd and $\Delta_{P_a}(t)$ does not have a Fox-Milnor factorization.

Assume that $a$ is an integer such that $a\equiv 5 \pmod{8}$. Then $\frac{a+1}{2}\equiv 3 \pmod{4}$, so $\frac{a+1}{2}$ has a prime divisor $p$ such that $p\equiv 3\pmod{4}$. Let 
\begin{center}
\(a=\prod\limits_{i=1}^{r}p_{i}^{e_{i}}\cdot  \prod\limits_{j=1}^{s}q_{i}^{f_{i}}\)
\end{center}
be the prime factorization of $a$ such that $p_{i} \equiv 1\pmod{4}$ and $q_{j} \equiv 3\pmod{4}$ for all $i$ and $j$. Since $a \equiv 1\pmod{4}$, $f=\sum_{j=1}^{s}f_{i}$ is even. Since $f$ is even and $a\equiv -1\pmod{p}$, we can observe that

\begin{center}\(\prod_{i=1}^{r} \bigl( \frac{p}{p_{i}} \bigr)^{e_{i}}
    \prod_{j=1}^{s} \bigl( \frac{p}{q_{j}} \bigr)^{f_{j}}=
    \prod_{i=1}^{r} \bigl( \frac{p_{i}}{p} \bigr)^{e_{i}}
    \prod_{j=1}^{s} (-1)^{f_{j}}\bigl( \frac{q_{j}}{p} \bigr)^{f_{j}}=(-1)^f \bigl( \frac{a}{p} \bigr)=\bigl( \frac{-1}{p} \bigr) = -1 \)
    \end{center}
     by using Lemma~\ref{lemma:QR}. There is a prime factor $d$ of $a$ such that  $\bigl( \frac{p}{d} \bigr)=-1$. By Proposition~\ref{prpn:Legendre}, $N_d^p$ is odd and $\Delta_{P_a}(t)$ does not have a Fox-Milnor factorization.
\end{proof}

\begin{remark}
The case $a\equiv 5 \pmod{12}$ is already proved by Lecuona, but Theorem~\ref{thm:main} gives an alternative proof. The case $a\equiv 3\pmod{8}$ is also proved in Theorem~\ref{thm:main2} below.
\end{remark}

\section{Method 2: Finding a self-reciprocal irreducible factor}
\label{sec:SR}
The goal of this section is to prove Theorem~\ref{thm:main2} which states that $P_a$ is not topologically slice if $a\equiv 3 \pmod{4}$. We will find an irreducible self-reciprocal factor of the mod $p$ reduction $\overline{\Phi}_d^p(t)$ of the $d$th cyclotomic polynomial $\Phi_d(t)$ for some $d$ and $p$. For this purpose, in Theorem~\ref{thm:SR}, we discuss equivalent conditions regarding the existence of a self-reciprocal irreducible factor of $\overline{\Phi}_d^p(t)$. Note that $\overline{\Phi}_d^p(t)$ has an irreducible self-reciprocal factor if and only if $\overline{\Phi}_d^p(-t)$ has an irreducible self-reciprocal factor. From this,  we can understand when $\overline{\Phi}_d^p(-t)$ has an irreducible self-reciprocal factor and apply  Proposition~\ref{prpn:Lecuona-observation}\eqref{item:Lecuona-observation-3} to conclude that $P_a$ is not slice if $a\equiv 3\pmod{4}$. 

We remark that Theorem~\ref{thm:SR} is  partly due to Wu, Yue and Fan \cite{Wu-Yue-Fan:2020-1}. In fact, Wu, Yue and Fan showed in \cite[Theorem 2.3]{Wu-Yue-Fan:2020-1} that the conditions \eqref{item:SR-1} and \eqref{item:SR-2} are equivalent under a weaker assumption on $p$ and $d$. 

To simplify our discussion, for a nonzero integer $n$ and a prime $q$, we denote $v_q(n)$ and $u_q(n)$ be the integers such that 
\[n = q^{v_q(n)}u_q(n)\]
and $u_q(n)$ is not divisible by $q$. The number $v_q(n)$ is usually called the \emph{$q$-adic valuation} of $n$. Note that $u_2(n)$ is odd and if $m$ divides $n$, then $u_q(m)$ divides $u_q(n)$. 

\begin{theorem}[{cf.\ \cite[Theorem 2.3]{Wu-Yue-Fan:2020-1}}]\label{thm:SR} Let $p$ and $d$ are distinct primes. Suppose that $d$ is odd. Then the following are equivalent. 

  \begin{enumerate}
    \item\label{item:SR-1}  $\overline{\Phi}_{d}^p(t)$ has a self-reciprocal irreducible factor.
        \item\label{item:SR-2} There exists an integer $w$ such that $p^{w} \equiv -1 \pmod{d}$.
    \item\label{item:SR-3}  $p^{u} \not\equiv 1 \pmod{d}$ where $u=u_2(\varphi(d))$ and $\varphi(d)$ denotes Euler's totient function.
  \end{enumerate}
\end{theorem}

\begin{proof}
As we mentioned, the equivalence between \eqref{item:SR-1} and \eqref{item:SR-2} is proved by Wu, Yue and Fan in \cite[Theorem 2.3]{Wu-Yue-Fan:2020-1}.

We first show that \eqref{item:SR-2} implies \eqref{item:SR-3}. Let $w$ be an integer such that $p^{w} \equiv -1 \pmod{d}$ and suppose that $p^{u} \equiv 1 \pmod{d}$ to get a contradiction. Let $s=\gcd(2w,u)$ so that there are integers $x, y \in \Z$ such that $s=2wx + uy$.  Observe that $p^{s} =p^{2wx+uy}\equiv 1 \pmod{d}$.
  Since $s$ is odd and $s$ divides $2w$, $s$ divides $w$. It follows that $p^{w} \equiv 1 \pmod{d}$ which is a contradiction. 

It remains to show that \eqref{item:SR-3} implies \eqref{item:SR-2}. Let $r = \mathrm{ord}(p,d)$ be the order of $p$ in $\Z_{d}^{\times}$. By Euler's theorem, $r$ divides $\varphi(d)$. If $r$ is odd, then $r$ divides $u=u_2(\varphi(d))$. This implies that  $p^{u} \equiv 1 \pmod{d}$ which is a contradiction. We can conclude that $r$ is even. Observe that $p^{r/2} \not\equiv 1 \pmod{d}$ and $(p^{r/2})^{2} \equiv 1 \pmod{d}$. Since $d$ is prime, $\Z_{d}$ is a field and hence $p^{r/2} \equiv -1 \pmod{d}$.
\end{proof}

\begin{theorem}\label{thm:main2}
For an odd integer $a\geq 3$, let $P_a$ be the $3$-stranded pretzel knot $P(a,-a-2,-\frac{(a+1)^2}{2})$. If $a\equiv 3 \pmod{4}$, then the Alexander polynomial $\Delta_{P_a}(t)$ does not have a Fox-Milnor factorization and hence $P_a$ is not topologically slice.
\end{theorem}

\begin{proof}
Let $a\geq 3$ be an odd integer such that $a\equiv 3\pmod{4}$. By Proposition~\ref{prpn:Lecuona-observation}\eqref{item:Lecuona-observation-3}, it suffices to find a prime divisor $p$ of $\frac{a+1}{2}$ and a divisor $d$ of $a$ or $a+2$ such that $\overline{\Phi}_{d}^p(-t)$ has a self-reciprocal irreducible factor. For this purpose, we will use Theorem \ref{thm:SR}.  Let $u = u_{2}(\varphi(a(a+2)))$. Note that $u$ is an odd integer.

  If $(\frac{a+1}{2})^{u} \equiv 1 \pmod{a(a+2)}$, then $2^{u} \equiv (a+1)^{u} \pmod{a(a+2)}$. Since $(a+1)^{2} \equiv 1 \pmod{a(a+2)}$ and $u$ is odd, 
  \[2^{u} \equiv (a+1)^u\equiv (a+1) \pmod{a(a+2)}.\]
   It follows that $2^{u} \equiv -1 \pmod{d}$ for any prime factor $d$ of $a+2$. Note that $2$ is a prime factor of $\frac{a+1}{2}$ since $a \equiv 3 \pmod{4}$. By Theorem \ref{thm:SR}, $\overline{\Phi}_d^2(-t)$ has a self-reciprocal irreducible factor for any prime divisor $d$ of $a+2$.

  If $(\frac{a+1}{2})^{u} \not\equiv 1 \pmod{a(a+2)}$, choose a prime factor $p$ of $\frac{a+1}{2}$ such that $p^u\not\equiv 1 \pmod{a(a+2)}$. By Chinese remainder theorem, there is a prime factor $d$ of $a(a+2)$ such that $p^{u} \not\equiv 1 \pmod{d^{l}}$ where $l=v_{d}(a(a+2))$. Since $u_{2}(\varphi(d^{l}))=d^{l-1}u_{2}(\varphi(d))$ is a divisor of $u$, $( p^{u_{2}(\varphi(d))} )^{d^{l-1}} \not\equiv 1 \pmod{d^{l}} $. By Proposition~\ref{prpn:PrimePower} given below, $p^{u_{2}(\varphi(d))} \not\equiv 1 \pmod{d}$. By Theorem \ref{thm:SR},  $\overline{\Phi}_d^p(-t)$ has a self-reciprocal irreducible factor. 
\end{proof}

It remains to prove the following elementary fact which was used in the above proof.
\begin{proposition}\label{prpn:PrimePower}
For a prime $d$, if $n \equiv 1 \pmod{d}$, then $n^{d^{l-1}} \equiv 1 \pmod{d^{l}}$.
\end{proposition}

\begin{proof}
  Write $n=md+1$ for some integer $m$. Then 
  \begin{center}
\(
    n^{d^{l-1}}=(md+1)^{d^{l-1}}  =1+ \sum_{k=1}^{d^{l-1}} m^{k}d^{k}\binom{d^{l-1}}{k}.
 \)
 \end{center}
 Let $k$ be an integer with $1\leq k\leq d^{l-1}$. Then $d^{l-1}$ divides $d^{v_d(k)}\binom{d^{l-1}}{k}$ since $k\binom{d^{l-1}}{k} = d^{l-1}\binom{d^{l-1}-1}{k-1}$. Since $v_d(k)<k$,   $d^{l}$ divies $d^{k}\binom{d^{l-1}}{k}$ for all such an integer $k$. Thus 
  $n^{d^{l-1}} \equiv 1 \pmod{d^{l}}$.
\end{proof}

\section{Discussion on the remaining cases}
\label{sec:examples}The remaining cases of Conjecture~\ref{conjecture:Lecuona} are when $a\equiv 1, 97 \pmod{120}$. In either case, there are examples satisfying the conditions \eqref{item:Lecuona-observation-2}, \eqref{item:Lecuona-observation-3} and \eqref{item:Lecuona-observation-4} of Proposition~\ref{prpn:Lecuona-observation}:
 \[a= 1081, 3577, 11257, 12457, 12841, 14617, 17521, 17881, \ldots.\]
 Hence the methods of Sections~\ref{section:method1} and \ref{sec:SR} do not work for these knots.

\bibliographystyle{amsalpha}
\renewcommand{\MR}[1]{}
\bibliography{research}

\end{document}